\documentclass[11pt]{amsart}
\usepackage{latexsym}
\usepackage{color}
\usepackage{enumerate}
\usepackage{amsfonts, amsthm, amsmath, amssymb}
\newtheorem{thm}{Theorem}[section]

\newtheorem{conjecture}[thm]{Conjecture}

\newtheorem{prop}[thm]{Proposition}
\newtheorem{cor}[thm]{Corollary}

 \usepackage{setspace}

\theoremstyle{remark}
    \newtheorem{rem}[thm]{Remark}

\theoremstyle{definition}
    \newtheorem{defn}[thm]{Definition}

\theoremstyle{question}
    
\theoremstyle{example}

\theoremstyle{examples}

\def\ds{\displaystyle}

\numberwithin{equation}{section}
\def\F {{\mathbb F}}
\def\Z {{\mathbb Z}}
\def\N{{\mathcal N}}
\def\Q {{\mathbb Q}}
\def\T {{\tilde T}}
\def\B {{\mathcal B}}
\def\C {{\mathcal C}}
\def\D {{\mathcal D}}
\def\H {{\mathcal H}}
\def\M {{\mathcal M}}
\def\h {{\mathfrak h}}
\def\a {{\mathfrak a}}
\def\l {\left\langle}
\def\r {\right\rangle}
\def\bip{{\mathrm{Bip}}}

\def\red#1 {\textcolor{red}{#1 }}
\def\blue#1 {\textcolor{blue}{#1 }}

\begin{document}

\title[Polynomial Cunningham Chains]{Polynomial Cunningham Chains}
\author{Lenny Jones}
\address{Department of Mathematics, Shippensburg University, Pennsylvania, USA}
\email[Lenny~Jones]{lkjone@ship.edu}
\date{\today}

\def\ds{\displaystyle}

\numberwithin{equation}{section}
\def\F {{\mathbb F}}
\def\Z {{\mathbb Z}}
\def\N{{\mathbb N}}
\def\Q {{\mathbb Q}}
\def\T {{\tilde T}}
\def\B {{\mathcal B}}
\def\C {{\mathcal C}}
\def\D {{\mathcal D}}
\def\H {{\mathcal H}}
\def\M {{\mathcal M}}
\def\h {{\mathfrak h}}
\def\a {{\mathfrak a}}
\def\l {\left\langle}
\def\r {\right\rangle}
\def\bip{{\mathrm{Bip}}}

\begin{abstract}
Let $\epsilon\in \{-1,1\}$. A sequence of prime numbers $p_1, p_2, p_3, \ldots$, such that $p_i=2p_{i-1}+\epsilon$ for all $i$, is called a {\it Cunningham chain} of the first or second kind, depending on whether $\epsilon =1$ or $-1$ respectively. If $k$ is the smallest positive integer such that $2p_k+\epsilon$ is composite, then we say the chain has length $k$. Although such chains are necessarily finite, it is conjectured that for every positive integer $k$, there are infinitely many Cunningham chains of length $k$. A sequence of polynomials $f_1(x), f_2(x), \ldots $, such that $f_i(x)\in \Z[x]$, $f_1(x)$ has positive leading coefficient, $f_i(x)$ is irreducible in $\Q[x]$, and $f_i(x)=xf_{i-1}(x)+\epsilon$ for all $i$, is defined to be a {\it polynomial Cunningham chain} of the first or second kind, depending on whether $\epsilon =1$ or $-1$ respectively. If $k$ is the least positive integer such that $f_{k+1}(x)$ is reducible over $\Q$, then we say the chain has length $k$.  In this article, for chains of each kind, we explicitly give infinitely many polynomials $f_1(x)$, such that $f_{k+1}(x)$ is the only term in the sequence $\{f_i(x)\}_{i=1}^{\infty}$ that is reducible. As a first corollary, we deduce that there exist infinitely many polynomial Cunningham chains of length $k$ of both kinds, and as a second corollary, we have that, unlike the situation in the integers, there exist infinitely many polynomial Cunningham chains of infinite length of both kinds.
\end{abstract}


\maketitle

\section{Introduction}\label{Intro}

We begin by giving the definitions of standard Cunningham chains \cite{Ribenboim}.
\begin{defn}\label{Def:Basic}
 Let $\epsilon\in \{-1,1\}$.  A sequence of prime numbers $p_1, p_2, p_3, \ldots$, such that $p_i=2p_{i-1}+\epsilon$ for all $i$, is called a {\it Cunningham chain} of the first or second kind, depending on whether $\epsilon =1$ or $-1$ respectively. If $k$ is the smallest positive integer such that $2p_k+\epsilon$ is composite, then we say the chain has length $k$.
\end{defn}

We define polynomial Cunningham chains in a similar manner.

\begin{defn}\label{Def:Poly}
  A sequence of polynomials $f_1(x), f_2(x), \ldots $, such that $f_i(x)\in \Z[x]$, $f_1(x)$ has positive leading coefficient, $f_i(x)$ is irreducible over $\Q$, and $f_i(x)=xf_{i-1}(x)+\epsilon$ for all $i$, is called a {\it polynomial Cunningham chain} of the first or second kind, depending on whether $\epsilon =1$ or $-1$ respectively. If $k$ is the least positive integer such that $f_{k+1}(x)$ is reducible over $\Q$, then we say the chain has length $k$.
\end{defn}

\begin{rem}
Note that if the restriction in Definition \ref{Def:Poly} that $f_1(x)$ have positive leading coefficient is dropped, then any chain of the first or second kind, with first term $f_1(x)$, produces a chain of the other kind with the exact same irreducibility properties by simply allowing the first term to be $-f_1(x)$.
\end{rem}

Although, using Fermat's little theorem, it is straightforward to show that Cunningham chains are necessarily finite, it is conjectured that for any positive integer $k$, there exist infinitely many Cunningham chains of length $k$. This conjecture follows from either Dickson's conjecture \cite{Dickson, Ribenboim} or Schinzel's hypothesis H \cite{Ribenboim, SchinzelSierpinski}, but it is unlikely to be proven unconditionally in the near future. Currently, the longest known Cunningham chain is of length 17. 
 In this article, we establish the analogous conjecture for polynomial Cunningham chains, and we also show that polynomial Cunningham chains can be infinite in length. More precisely, we prove the following.

\begin{thm}\label{Thm:1} For $\epsilon \in \left\{-1,1\right\}$, and a given polynomial $f_1(x)$, we define a sequence $\{f_i(x)\}_{i=1}^{\infty}$ of polynomials by $f_i(x)=xf_{i-1}(x)+\epsilon$ for $i\ge 2$.
\begin{enumerate}
\item Let $\epsilon=1$. Let $m\ge 2$ and $k\ge 1$ be integers. Define \[f_1(x):=m^2x^{k+3}+mx^{k+2}+mx^{k+1}+\cdots +mx+1.\] Then $f_i(x)$ is reducible over $\Q$ if and only if $i=k+1$.
\item Let $\epsilon=-1$. Let $k$ and $m$ be positive integers with $m^2>k+1$. Define \[f_1(x):=m^2x-(m^2-k).\] Then $f_i(x)$ is reducible over $\Q$ if and only if $i=k+1$.
    \end{enumerate}
\end{thm}
Corollary \ref{Cor:Cun-Lengthk} is immediate from Theorem \ref{Thm:1}, and Corollary \ref{Cor:Cun-Infinite Length} follows from Theorem \ref{Thm:1} by considering sequences $\{g_i(x)\}_{i=1}^{\infty}$, with $g_1(x):=f_{k+2}(x)$, and $g_i(x)=xg_{i-1}(x)+\epsilon$ for $i\ge 2$.
\begin{cor}\label{Cor:Cun-Lengthk}
  For every positive integer $k$, there exist infinitely many polynomial Cunningham chains (of both kinds) of length $k$.
\end{cor}

\begin{cor}\label{Cor:Cun-Infinite Length}
 There exist infinitely many polynomial Cunningham chains (of both kinds) of infinite length.
 \end{cor}

\section{Preliminaries}\label{Sec:Prelim}
We begin this section with some more definitions and notation, and we let $f(x)\in \Z[x]$ throughout this section.

\begin{defn}\label{Def:Recip}
The {\it reciprocal} of $f(x)$ is defined to be the polynomial
$\ds \widetilde{f}(x):=x^{\deg f}f\left(\frac{1}{x}\right)$. 
We say that $f(x)$ is {\it reciprocal} if $f(x)=\pm \widetilde{f}(x)$, and {\it nonreciprocal} otherwise.
\end{defn}
\begin{defn}
Suppose $f(0)\ne 0$, and that $f(x)$ factors over $\Q$ into irreducibles as $g_1(x)g_2(x)\cdots g_k(x)$, where $g_i(x)$ is reciprocal with positive leading coefficient, exactly when $0\le i \le j$. Then $g_1(x)g_2(x)\cdots g_j(x)$ is called the {\it reciprocal part of $f$} and $g_{j+1}(x)\cdots g_{k}(x)$ is called the {\it nonreciprocal part of $f$}. 
\end{defn}

If $f(0)\ne 0$, then it is clear from Definition \ref{Def:Recip} that $\deg f=\deg \widetilde{f}$ and $\widetilde{\widetilde{f}}(x)=f(x)$. Then, in this situation, $f(x)=g(x)h(x)$ if and only if $\widetilde{f}(x)=\widetilde{g}(x)\widetilde{h}(x)$. Therefore, we have the following:
\begin{prop}\label{Prop:RecipIrreduc}
 Suppose that $f(0)\ne 0$. Then $f(x)$ is irreducible over $\Q$ if and only if $\widetilde{f}(x)$ is irreducible over $\Q$.
\end{prop}

The following theorem due to Fried and Schinzel \cite{FriedSchinzel}, which we state without proof, is needed to establish our results.
 We let $\exp_2(\alpha)$ denote $\exp(\exp(\alpha))$ for any expression $\alpha$.

\begin{thm}\label{Thm:SchinzelQuad1}
Let $a,b,c,d$ be any nonzero integers, $m>n>p$ any positive integers, and assume that $q(x)=ax^m+bx^n+cx^p+d$ is not the product of two binomials. Then the nonreciprocal part of $q(x)$ is reducible if and only if one of the following cases holds:
\begin{enumerate}
\item \label{Thm:SQ1-1} $q(x)$ can be divided into two parts which have a nonreciprocal common factor
\item \label{Thm:SQ1-2} $q(x)$ can be represented in one of the three forms in {\rm ($*$)} below:\\

\hspace*{-.7in}$(*)\left\{\begin{array}{l}
\xi(U^3+V^3+W^3-3UVW)\\=\xi(U+V+W)(U^2+V^2+W^2-UV-UW-VW),\\\\

\xi(U^2-4TUVW-T^2V^4-4T^2W^4)\\=\xi(U-TV^2-2TVW-2TW^2)(U+TV^2-2TVW+2TW^2),\\\\

\xi(U^2+2UV+V^2-W^2)=\xi(U+V+W)(U+V-W),
\end{array}\right.$\\

\noindent where $T,U,V,W\in \Q[x]$ are monomials, $\xi\in \Q$, and the factors appearing on the right hand side of each equation in {\rm ($*$)} are not reciprocal
\item \label{Thm:SQ1-3} $m=vm_1$, $n=vn_1$, $p=vp_1$, where $v>1$,
\[m_1<\exp_2(3\cdot 2^{a^2+b^2+c^2+d^2+2}\log (a^2+b^2+c^2+d^2)),\]
and the non-reciprocal part of $ax^{m_1}+bx^{n_1}+cx^{p_1}+d$ is reducible.
\end{enumerate}

\end{thm}

\section{Proof of Theorem \ref{Thm:1}}
\begin{proof}[Proof of Theorem \ref{Thm:1}]
For part {\it (1)}, we have by induction that 
 \[f_n(x)=m^2x^{n+k+2}+mx^{n+k+1}+mx^{n+k}+\cdots +mx^{n}+x^{n-1}+x^{n-2}+\cdots +x+1,\]
for all $n\ge 1$. Since
\[f_{k+1}(x)=\left(mx^{k+1}+x^k+\cdots +x+1\right)\left(mx^{k+2}+1\right),\]
we see that $f_n(x)$ is reducible when $n=k+1$. To show that $f_n(x)$ is irreducible over $\Q$ for all $n\ne k+1$, we show that \[\widetilde{f_n}(x)=x^{n+k+2}+x^{n+k+1}+\cdots + x^{k+3}+mx^{k+2}+\cdots +mx+m^2\] is irreducible over $\Q$ for all $n\ne k+1$, which is equivalent by Proposition \ref{Prop:RecipIrreduc}.
We first claim that all zeros of $\widetilde{f_n}(x)$ are in $|z|>1$.
Consider the polynomial
\[F_n(x):=(x-1)\widetilde{f_n}(x)=x^{n+k+3}+(m-1)x^{k+3}+(m^2-m)x-m^2,\]
and let $\alpha$ be a zero of $F_n(x)$. If $|\alpha|<1$, then
\begin{align*}
  m^2&=\left|\alpha^{n+k+3}+(m-1)\alpha^{k+3}+(m^2-m)\alpha\right|\\
  &\le |\alpha|^{n+k+3}+(m-1)|\alpha|^{k+3}+(m^2-m)|\alpha|\\
  &<m^2,
\end{align*}
which is impossible.
 Hence, $|\alpha|\ge 1$. If $|\alpha|=1$, then $\alpha=e^{i\theta}=\cos(\theta)+i\sin(\theta)$, for some $\theta\in [0,2\pi)$.
 Thus,\[\cos\left((n+k+3)\theta\right)+(m-1)\cos\left((k+3)\theta\right)+(m^2-m)\cos(\theta)=m^2,\]
 which implies that $\theta=0$, and so $\alpha=1$. This establishes the claim that all zeros of $\widetilde{f_n}(x)$ are in $|z|>1$. It follows that the nonreciprocal part of $F_n(x)$ is $\widetilde{f_n}(x)$.

  We now use Theorem \ref{Thm:SchinzelQuad1} with $q(x):=F_n(x)$ to show that $\widetilde{f_n}(x)$ is irreducible when $n\ne k+1$. We see easily that
\begin{align*}
F_n(x)&=x^{n+k+3}+(m-1)x^{k+3}+(m^2-m)x-m^2\\
&=(x^{\ell_1}+r_1)(x^{\ell_2}+r_2)\\
&= x^{\ell_1+\ell_2}+r_2x^{\ell_1}+r_1x^{\ell_2}+r_1r_2
\end{align*}
 is impossible by comparing coefficients. Thus, $F_n(x)$ is not the product of two binomials.

Next, we show that $F_n(x)$ cannot be any of the forms in $(*)$ in Theorem \ref{Thm:SchinzelQuad1}.
First, assume that \[F_n(x)=\xi(U^3+V^3+W^3-3UVW).\] Then $-3\xi UVW=(m^2-m)x$, and hence two of the three terms $U^3,V^3,W^3$ are constant, which is a contradiction.

Suppose next that \[F_n(x)=\xi(U^2-4TUVW-T^2V^4-4T^2W^4).\] If $\xi>0$, then $-\xi T^2V^4$ and $-4\xi T^2W^4$ are negative terms, which is impossible. Thus, $\xi<0$, and so $\xi U^2=-m^2$. The parity of the exponents implies that $-4\xi TUVW =(m^2-m)x$. Therefore, exactly two of $T$, $V$ and $W$ are constants. Since we have only the two possibilities
\begin{equation}\label{Eq:Poss1} -\xi T^2V^4=x^{n+k+3}\quad \mbox{and} \quad -4\xi T^2W^4=(m-1)x^{k+3},\quad \mbox{or}\end{equation}
\begin{equation}\label{Eq:Poss2} -\xi T^2V^4=(m-1)x^{k+3}\quad \mbox{and} \quad -4\xi T^2W^4=x^{n+k+3},\end{equation}
it follows that $V$ and $W$ must be constants. Then, since $-4\xi TUVW =(m^2-m)x$, we have that $T$ is divisible by $x$, but not $x^2$. However, by comparing exponents, we see then that both possibilities, (\ref{Eq:Poss1}) and (\ref{Eq:Poss2}), are impossible since $n+k+3\ge5$.

Now suppose that \[F_n(x)=\xi(U^2+2UV+V^2-W^2).\]
If $\xi<0$, then the terms $\xi U^2$ and $\xi V^2$ are both negative, which is impossible. Thus, $\xi>0$, and so $-\xi W^2=-m^2$. Then both $U$ and $V$ must be divisible by $x$, which implies that $U^2$, $UV$ and $V^2$ are divisible by $x^2$. But this contradicts the fact that $F_n(x)$ contains the term $(m^2-m)x$.

Also, the fact that $F_n(x)$ contains the linear term $(m^2-m)x$ implies that case {\it (\ref{Thm:SQ1-3})} does not apply in Theorem \ref{Thm:SchinzelQuad1}.

Next we consider if, and when, $F_n(x)$ can be divided into two parts which have a common nonreciprocal factor. There are three cases to check: $(g_1, g_2)=$
\begin{enumerate}[(i)]
  \item \label{Case:1} $\left(x^{n+k+3}+(m-1)x^{k+3},\quad (m^2-m)x-m^2\right)$
  \item \label{Case:2} $\left(x^{n+k+3}+(m^2-m)x,\quad (m-1)x^{k+3}-m^2\right)$
  \item \label{Case:3} $\left(x^{n+k+3}-m^2,\quad (m-1)x^{k+3}+(m^2-m)x\right)$.
\end{enumerate}
In case (\ref{Case:1}), it is easy to see that $g_1$ and $g_2$ have no common nonreciprocal factor since $m/(m-1)$ is not a zero of $g_1$.

In case (\ref{Case:2}), $g_1/x$ is irreducible over $\Q$ by Eisenstein's criterion using any prime divisor of $m$. Thus, by considering degrees, we can rule out every possibility except $n=1$. But, when $n=1$, we see that $g_2/(m-1)\ne \pm g_1/x$, and hence $g_1$ and $g_2$ have no common nonreciprocal factor in this case.

For case (\ref{Case:3}), let $h$ be a common factor of $g_1$ and $g_2$. Then $h$ divides
\[x^n\frac{g_2}{m-1}-g_1=m\left(x^{n+1}+m\right),\] so that $h$ divides
\[x^{n+1}+m-\frac{g_2}{(m-1)x}=\left\{\begin{array}{cc}
-x^{n+1}\left(x^{k-n+1}-1\right) & \mbox{ if $n<k+1$}\\
x^{k+2}\left(x^{n-k-1}-1\right) & \mbox{ if $n>k+1$}
\end{array} \right.\]
In either case, we see that $h$ is reciprocal, which implies that $\widetilde{f_n}(x)$ is irreducible for all $n\ne k+1$.
This completes the proof in the case of chains of the first kind.

To establish part {\it (2)} of the theorem, we have by induction that
\[f_n(x)=m^2x^{n}-(m^2-k)x^{n-1}-x^{n-1}-\cdots +x-1,\]
for all $n\ge 1$.
Since $f_{k+1}(1)=0$, we see that $f_n(x)$ is reducible when $n=k+1$. To show that $f_n(x)$ is irreducible over $\Q$ for all $n\ne k+1$, it is enough, by Proposition \ref{Prop:RecipIrreduc}, to show that \[-\widetilde{f_n}(x)=x^{n}+x^{n-1}+\cdots + x^{2}+(m^2-k)x-m^2\] is irreducible over $\Q$ for all $n\ne k+1$.
Consider the polynomial
\[F_n(x):=-(x-1)\widetilde{f_n}(x)=x^{n+1}+(m^2-k-1)x^{2}-(2m^2-k)x+m^2.\]
 We claim that $-\widetilde{f_n}(x)$ is the nonreciprocal part of $F_n(x)$ when $n\ne k+1$. Suppose, by way of contradiction, that $\alpha$ and $1/\alpha$ are both zeros of $F_n(x)$, with $\alpha\ne 1$. Then
\begin{equation}\label{Eq:alpha}
-\alpha^{n+1}=(m^2-k-1)\alpha^2-(2m^2-k)\alpha+m^2 \quad \mbox{ and}
\end{equation}
\begin{equation}\label{Eq:1/alpha}
  -\frac{1}{\alpha^{n+1}}=\frac{m^2-k-1}{\alpha^2}-\frac{2m^2-k}{\alpha}+m^2.
\end{equation}
Substituting the expression for $-\alpha^{n+1}$ from (\ref{Eq:alpha}) into (\ref{Eq:1/alpha}), rearranging and factoring, gives $(\alpha-1)^2g(\alpha)=0$, where
\[g(x)=\left(m^4-m^2k-m^2\right)x^2-\left(2m^4-2m^2k+k^2+k\right)x+m^4-m^2k-m^2.\] Note that $\alpha$ and $1/\alpha$ are distinct positive real zeros of $g(x)$ since $m^2>k+1$. However, by Descartes' rule of signs, $F_n(x)$ has two positive real zeros counting multiplicities, and since $-\widetilde{f_n}(1)=0$ only when $n=k+1$, it follows that $F_n(x)$ has exactly one positive real zero $\beta\ne 1$ when $n\ne k+1$. This contradiction establishes the claim.

We now use Theorem \ref{Thm:SchinzelQuad1} with $q(x):=F_n(x)$ to show that $\widetilde{f_n}(x)$ is irreducible when $n\ne k+1$. Since, as in the case of chains of the first kind, it is straightforward to show that $F_n(x)$ is not the product of two binomials, and that $F_n(x)$ cannot be any of the forms in $(*)$ in Theorem \ref{Thm:SchinzelQuad1}, we omit the details. Also, the fact that $F_n(x)$ contains a linear term implies that case {\it (\ref{Thm:SQ1-3})} does not apply in Theorem \ref{Thm:SchinzelQuad1}.

Next we consider if, and when, $F_n(x)$ can be divided into two parts which have a common nonreciprocal factor. There are three cases to check: $(g_1, g_2)=$
\begin{enumerate}[(i)]
  \item \label{Case:1A} $\left(x^{n+1}+(m^2-k-1)x^2,\quad -(2m^2-k)x+m^2\right)$
  \item \label{Case:2A} $\left(x^{n+1}-(2m^2-k)x,\quad (m^2-k-1)x^2+m^2\right)$
  \item \label{Case:3A} $\left(x^{n+1}+m^2,\quad (m^2-k-1)x^{2}-(2m^2-k)x\right)$.
\end{enumerate}
It is easy to see in each of these cases, by examining the zeros of $g_2$, that $g_1$ and $g_2$ have no common nonreciprocal factor.
Thus, $f_n(x)$ is irreducible if and only if $n\ne k+1$, which completes the proof of the theorem.

\end{proof}

\section{Concluding Remarks}

 For each $\epsilon\in \left\{-1,1\right\}$, Theorem 3 in \cite{SchinzelTri2001}, rather than Theorem \ref{Thm:SchinzelQuad1}, can be used to establish the existence of infinitely many polynomials $f_1(x)$, such that there is exactly one reducible polynomial in the sequence $\{f_i(x)\}_{i=1}^{\infty}$, where $f_i(x)=xf_{i-1}(x)+\epsilon$ for $i\ge 2$. However, the drawback is that the polynomials $f_1(x)$ have no set form using this approach, as do the polynomials $f_1(x)$ given in Theorem \ref{Thm:1}, and so they cannot be given explicitly. Along these lines we make the following conjecture.
\begin{conjecture}
  Let $k\ge 1$ be an integer. Let $t=2\lceil\frac{k+1}{2}\rceil+1$, and let $m=2\left(\frac{2^t+1}{3}\right)$.
   Then
  \[x^j+x^{j-1}+\cdots +x+m,\] where $j\ge t-k$,
  is reducible over $\Q$ if and only if $j=t$.
\end{conjecture}

Corollary \ref{Cor:Cun-Infinite Length} can also be proven easily without the use of Theorem \ref{Thm:SchinzelQuad1}. In the case of $\epsilon=1$, let $p$ be a prime, and define $f_1(x):=px+1$. In the case of $\epsilon=-1$, let $c$ be any positive integer, define $f_1(x):=x-c$, and use the following result due to Alfred Brauer \cite{ABrauer}.
 \begin{thm}\label{Thm:Brauer}
   Let $f(x)=x^n-a_{n-1}x^{n-1}-a_{n-2}x^{n-2}-\cdots -a_1x-a_0$. If $a_{n-1}\ge \cdots \ge a_0>0$, then $f(x)$ is irreducible over $\Q$.
 \end{thm}


\begin{thebibliography}{HD}


\bibitem[B]{ABrauer}
A. Brauer, \emph{On algebraic equations with all but one root in the interior of the unit circle}, Math. Nachr., 4 (1951), 250--257.

\bibitem[D]{Dickson}
L. E. Dickson, \emph{A new extension of Dirichlet's theorem on prime numbers}, Messenger of Mathematics, 33 (1904), 155--161.

\bibitem[FS]{FriedSchinzel}
 M. Fried and A. Schinzel,
\emph{Reducibility of quadrinomials}, Acta Arith., 21 (1972),
153--172.

\bibitem[R]{Ribenboim}
P. Ribenboim, \emph{The new book of prime number records}, 3rd ed., Springer-Verlag, New York, 1996.

\bibitem[S]{SchinzelTri2001}
A. Schinzel, \emph{On reducible trinomials. III.}, Period. Math. Hungar. 43 (2001), no. 1-2, 43-–69.


\bibitem[SS]{SchinzelSierpinski}
A. Schinzel and W. W.~Sierpi\'{n}ski, \emph{Sur certaines hypotheses concernment les nombres premiers}, Acta. Arith., 4 (1958), 185--208.  Erratum 5 (1958).


\end{thebibliography}
\end{document}